\documentclass[11pt]{amsart}
\usepackage[english]{babel}
\usepackage{amssymb}
\usepackage{color}
\newtheorem{thm}{Theorem}

\newtheorem{exam}{Example}

\begin{document}

	\author{E. Y. Emelyanov$^{1}$}
	\address{$^1$ Middle East Technical University, 06800 Ankara, Turkey}
	\email{eduard@metu.edu.tr}
	\author{S. G. Gorokhova$^2$}
	\address{$^2$ Sobolev Institute of Mathematics, 630090 Novosibirsk, Russia}
	\email{lanagor71@gmail.com}
	
\bigskip
		
	\keywords{universally complete vector lattice, $uo$-convergence}
	\subjclass[2010]{46A40}
	\date{\today}
	
\title{Unbounded Order Convergence and Universal Completions}

\begin{abstract}
We characterize vector lattices in which unbounded order convergence is eventually order bounded. Among other things, the characterization provides a solution to
\cite[Probl.23]{Az}.
\end{abstract}

\maketitle
\date{\today}

\section{Preliminaries}

Throughout the paper, $X$ stands for a vector lattice and all vector lattices are assumed to be real and Archimedean.
We refer to \cite{AB,Kus,GTX} for unexplained terminology and standard facts on vector lattice theory. 

We recall a few standard definitions and results related to vector lattices. $X$ is said to be Dedekind ($\sigma$-Dedekind) complete
if every order bounded (countable) subset of $X$ has a supremum. A Dedekind complete ($\sigma$-Dedekind complete) $X$ is said to be 
universally ($\sigma$-universally) complete if every pairwise disjoint (countable) subset of $X_+$ has a supremum. 
Every universally complete vector lattices has a weak unit. It is well known that $X$ possesses a unique up to lattice isomorphism 
Dedekind (universal) completion, which will be denoted by $X^\delta$ (by $X^u$). Dealing with the completions, 
we always suppose that $X\subseteq X^\delta\subseteq X^u$, whereas $X^\delta$ sits as an ideal in $X^u$.

A sublattice $Y$ of $X$ is called regular if $y_\alpha\downarrow 0$ in $Y$ implies $y_\alpha\downarrow 0$ in $X$.
$Y$ is said to be order dense if for every $0\ne x\in X_+$ there exists $0\ne y\in Y_+$ such that $y\le x$.
It is well known that ideals and order dense sublattices are regular. Furthermore, $X$ is atomic iff 
it is lattice isomorphic to an order dense sublattice of ${\mathbb{R}^A}$ (cf. \cite[Thm.1.78]{AB}). 

A net $(x_\alpha)_{\alpha\in A}$ in $X$ $o$-converges to $x$ if there exists a net $(y_\gamma)_{\gamma\in\Gamma}$ in $X$ satisfying 
$y_\gamma\downarrow 0$ and for each $\gamma\in\Gamma$ there is $\alpha_\gamma\in A$ with $|x_\alpha-x|\leq y_\gamma$ for all $\alpha\geq\alpha_\gamma$. 
In this case we write $x_\alpha\xrightarrow{o}x$. This definition is used e.g. in \cite{Kus,GTX}. In some of the literature (cf. \cite{AB}) 
a slightly different definition of the order convergence is used, namely a net $(x_\alpha)_{\alpha\in A}$ in $X$ is 
said to be $o$-convergent to $x$ if there exists a net $(z_\alpha)_{\alpha\in A}$ in $X$ such that $z_\alpha\downarrow 0$ and $|x_\alpha-x|\leq z_\alpha$ 
for all $\alpha$. Notice that both notions coincide in the case of order bounded nets in a Dedekind complete vector lattice (cf. \cite[Rem.2.2]{GTX} ).
We refer to \cite{AS} for further discussion of definitions of $o$-convergence. It should be noted that $o$-convergence in $X$ is never topological 
unless $\dim(X)<\infty$ \cite[Thm.1]{DEM} (cf. also \cite[Thm.2]{Gor}).   

A net $x_\alpha$ is $X$ is unbounded order convergent (shortly, $uo$-convergent) to $x\in X$ if $|x_\alpha-x|\wedge y\xrightarrow{o}0$ 
for every $y\in X_+$. In this case we, write $x_\alpha\xrightarrow{uo}x$. Following Nakano \cite{Na}, $uo$-convergence is considered 
as a natural generalization of convergence almost everywhere (see \cite{GTX,LC,Az,Ta} and references therein). Clearly,
$o$-convergence agrees with eventually order bounded $uo$-convergence. By \cite[Thm.3.2]{GTX}, $uo$-convergence passes freely 
between $X$, $X^\delta$, and $X^u$. It was shown in \cite[Cor.3.5]{GTX} that if $X$ has a weak unit $u$, then for a net $x_\alpha$ in $X$, 
$x_\alpha\xrightarrow{uo}x\Leftrightarrow |x_\alpha-x|\wedge u\xrightarrow{o}0$. It was also proved in \cite[Cor.3.12]{GTX} that any 
$uo$-null sequence in $X$ is $o$-null in $X^u$. Accordingly to \cite[Ex.2.6]{LC}, it is no longer true for nets in $\ell_\infty$.
Theorem \ref{thm1} shows that all $uo$-null nets in $X$ are $o$-null in $X^u$ only in the trivial case $\dim(X)<\infty$.
For further purposes, we include the following modification of \cite[Ex.2.6]{LC}. 
\begin{exam}\label{exam-main-0} 
A net $(x_{(n,m)})_{(n,m)\in\mathbb{N}^2}$ in $X=c_{00}$ is defined by 
\[
x_{(n,m)}(k):=
\begin{cases}
		0 &\text{if } k\ne n\wedge m,\\
		n\vee m &\text{if } k=n\wedge m.
\end{cases}
\]
Shortly, $x_{(n,m)}=(n\vee m)\cdot\mathbb{I}_{\{n\wedge m\}}$.
Since $\lim\limits_{(n,m)\to\infty}x_{(n,m)}(k)=0$ for every $k\in\mathbb{N}$, then $x_{(n,m)}\xrightarrow{uo}0$
$($e.g., by \cite[Prop.1.]{DEM}$)$. Suppose $x_{(n,m)}$ is eventually 
order bounded by some $y\in\mathbb{R}^\mathbb{N}$. Then there exists $(n_0,m_0)\in\mathbb{N}^2$ such that
$$
  y\ge x_{(n,m)} \ \ \ \ \ (\forall (n,m)\ge(n_0,m_0)).
$$ 
Since $n\wedge m_0=m_0$ and $(n,m_0)\ge(n_0,m_0)$ for $n\ge n_0\vee m_0$, then
$$
  y(m_0)\ge x_{(n,m_0)}(m_0)=n\vee m_0=n \ \ \ \ \ (\forall n\ge n_0\vee m_0)
$$
which is impossible. Therefore, the $uo$-null net $x_{(n,m)}$ is not eventually order bounded in 
$\mathbb{R}^\mathbb{N}$ and hence is not $o$-convergent in $X^u=(c_{00})^u=\mathbb{R}^\mathbb{N}$.
\end{exam}
Although the $uo$-convergence is not topological in most of important cases (e.g., in $L_1[0,1]$ and in $C[0,1]$), 
it is topological in all atomic vector lattices \cite[Thm.2]{DEM}.
	
A net $x_\alpha$ in $X$ is called $o$-Cauchy ($uo$-Cauchy) if the double net $(x_\alpha-x_\beta)_{(\alpha,\beta)}$ $o$-converges ($uo$-converges) to $0$.
It was noticed in \cite[Lm.2.1]{LC} with a reference to \cite[Prop.5.7]{GTX} that every order bounded positive increasing net in $X$ is $o$-Cauchy.
A net $x_\alpha$ in a Dedekind complete vector lattice with a weak unit $u$ is $uo$-Cauchy iff $\inf_\alpha\sup_{\beta,\gamma\ge\alpha}|x_\beta -x_\gamma|\wedge u=0$ 
\cite[Lm.2.7]{LC}. It is well known that the completeness with respect to the $o$-convergence is equivalent to the Dedekind completeness.
By \cite[Cor.3.12]{GTX}, a sequence $x_n$ in X is $uo$-Cauchy in $X$ iff it is $o$-convergent in $X^u$. In the same paper, authors proved that 
a sequence in a $\sigma$-universally complete vector lattice is $uo$-Cauchy iff it is $o$-convergent \cite[Thm.3.10]{GTX}. Theorem \ref{thm1}
shows that there is no net-version of \cite[Thm.3.10]{GTX} unless $\dim(X)<\infty$. 

In \cite[Thm.3.9]{Ta} (see, also \cite[Thm.28]{Az}), it was shown that a vector lattice is $\sigma$-universally complete iff it is sequentially $uo$-complete.
It was also proved in \cite[Thm.17]{Az} that the $uo$-completeness is equivalent to the universal completeness.
Therefore, there is no need in considering $uo$-completions (sequential $uo$-completions) of vector lattices.

\section{Main result}

We begin with the following generalization of Example \ref{exam-main-0}. Given a nonempty subset $A\subset X$,
$pr_{A}$ stands for the band projection in $X^u$ onto the band in $X^u$ generated by $A$.

\begin{exam}\label{exam-main-1}
In any infinite-dimensional Archimedean vector lattice $X$ there exists a $uo$-null net which is not eventually order bounded in $X^u$. 

As $\dim(X)=\infty$, there is a sequence $e_n$ of pairwise disjoint positive nonzero elements of $X$. Let \ $\mathbb{N}^2$ be the coordinatewise directed set
of pairs of naturals. A net in $X$ is defined via $x_{(n,m)}=(n\vee m)\cdot e_{n\wedge m}$. Since $\{x_{(n,m)}:(n,m)\in\mathbb{N}^2\}\subseteq B_{\{e_k:k\in\mathbb{N}\}}$ and
$$
  \lim\limits_{(n,m)\to\infty}pr_{\{e_k\}}(x_{(n,m)})=\lim\limits_{(n,m)\to\infty}(n\vee m)pr_{\{e_k\}}(e_{n\wedge m})=0 \ \ \ \ (\forall k\in\mathbb{N}),
$$ 
then $x_{(n,m)}\xrightarrow{uo}0$ as $(n,m)\to\infty$ $($e.g., it can be seen by use of \cite[Cor.3.5.]{GTX} for a weak unit $u$ in $X^u$ s.t. $u\wedge e_k=e_k$ for all $k$$)$. 
If $x_{(n,m)}$ is eventually order bounded by some $y\in X^u$, then for some $(n_0,m_0)\in\mathbb{N}^2$ we have 
$$
  y\ge x_{(n,m)} \ \ \ \ \ (\forall (n,m)\ge(n_0,m_0)).
$$ 
Since $n\wedge m_0=m_0$ and $(n,m_0)\ge(n_0,m_0)$ for $n\ge n_0\vee m_0$, then
$$
  y\ge x_{(n,m_0)}=(n\vee m_0)\cdot e_{n\wedge m_0}=(n\vee m_0)\cdot e_{m_0}=n\cdot e_{m_0}>0 \ \ \ \ \ (\forall n\ge n_0\vee m_0)
$$
which is impossible. Therefore, the net \ $x_{(n,m)}$ is not eventually order bounded in $X^u$.
\end{exam}

\begin{thm}\label{thm1} Let $X$ be an Archimedean vector lattice. TFAE:\\
$(1)$\ $\dim(X)<\infty$$;$\\ 
$(2)$\ every $uo$-Cauchy net in $X$ is eventually order bounded in $X^u$$;$\\ 
$(3)$\ every $uo$-Cauchy net in $X$ is $o$-convergent in $X^u$$;$\\ 
$(4)$\ every $uo$-null net in $X$ is $o$-null in $X^u$$;$\\ 
$(5)$\ every $uo$-null net in $X$ is eventually order bounded in $X^u$$;$\\
$(6)$\ every $uo$-convergent net in $X$ is eventually order bounded in $X^u$$;$\\
$(7)$\ every $uo$-convergent net in $X$ is eventually order bounded in $X$$.$
\end{thm}

\begin{proof}
$(1)\Rightarrow (2)$, $(4)\Rightarrow(5)\Leftrightarrow(6)$, and $(7)\Rightarrow(6)$ are trivial.\\ 
$(2)\Rightarrow (3)$: Suppose $x_\alpha$ is $uo$-Cauchy in $X$. Then $x_\alpha$ is $uo$-Cauchy in $X^u$
by \cite[Thm.3.2]{GTX}, because $X$ is regular in $X^u$. It follows from \cite[Thm.17]{Az} that
$x_\alpha\xrightarrow{uo}y$ for some $y\in X^u$. Since $x_\alpha$ is eventually order bounded in $X^u$ by the assumption,
then $x_\alpha\xrightarrow{o}y$.\\ 
$(3)\Rightarrow (4)$ follows since every $uo$-null net is $uo$-Cauchy, $o$-convergent implies $uo$-convergent, and the
$uo$-limit of any $uo$-convergent net is unique.\\ 
$(5)\Rightarrow (1)$ is Example \ref{exam-main-1}.\\
$(6)\Rightarrow (7)$ follows from the equivalence $(6)\Leftrightarrow(1)$ because $(1)\Rightarrow(7)$ is obvious.
\end{proof}

The equivalence $(1)\Leftrightarrow(7)$ of Theorem \ref{thm1} justifies use of term ``{\it unbounded order convergence}''\
for the $uo$-convergence because the $uo$-conver\-gen\-ce for nets in $X$ is automatically eventually order bounded only if $X$ is finite-dimensional.

The following question ``{\it suppose that $X$ is an arbitrary Dedekind complete but not universally complete vector lattice. 
Is there a $uo$-Cauchy net in $X$ that fails to be $o$-convergent in $X^u$?}'' was asked in \cite[Prob.23]{Az}. 
Since $X\ne X^u$ implies $\dim(X)=\infty$, the equivalence $(1)\Leftrightarrow(3)$ of Theorem \ref{thm1} 
gives a positive answer to this question for an arbitrary non-universally complete Archimedean vector lattice $X$.

\end{document}